\documentclass[11pt]{article}

\usepackage{amsmath}
\usepackage{amsfonts}
\usepackage{amsthm}
\usepackage{amssymb}
\usepackage{MnSymbol}
\usepackage{bbm}
\usepackage{cite}
\usepackage{enumerate}
\usepackage{authblk}

\newcommand{\lam}{\lambda}
\newcommand{\eps}{\epsilon}

\newcommand{\Q}{\mathbb{Q}}
\newcommand{\N}{\mathbb{N}}
\newcommand{\E}{\mathbb{E}}
\renewcommand{\P}{\mathbb{P}}
\newcommand{\curlyG}{\mathcal{G}}
\newcommand{\curlyH}{\mathcal{H}}
\newcommand{\nbr}{\mathcal{N}}
\newcommand{\normal}{\mathcal{N}}
\renewcommand{\Pr}{\text{Pr}}
\DeclareMathOperator{\Var}{Var}

\DeclareMathOperator{\Binom}{Binom}
\DeclareMathOperator{\Pois}{Pois}

\newcommand{\toD}{\stackrel{d}{\to}}

\newtheorem{theorem}{Theorem}[section]
\newtheorem{lemma}[theorem]{Lemma}

\newtheorem{proposition}[theorem]{Proposition}
\newtheorem{definition}[theorem]{Definition}
\newtheorem{conjecture}[theorem]{Conjecture}
\newtheorem{remark}[theorem]{Remark}

\title{Reconstruction and Estimation in the Planted Partition Model}
\author[1,2]{Elchanan Mossel\thanks{
Supported by NSF grant DMS-1106999 and DOD ONR grant N000141110140}}
\author[1]{Joe Neeman${}^{\small *}$}
\author[1]{Allan Sly}
\affil[1]{Department of Statistics, UC Berkeley}
\affil[2]{Department of Computer Science, UC Berkeley}

\begin{document}
\maketitle

\begin{abstract}
The planted partition model (also known as the stochastic blockmodel)
is a classical cluster-exhibiting random graph model that has been
extensively studied in statistics, physics, and computer
science.
In its simplest form, the planted partition model is a model for
random graphs on $n$ nodes
with two equal-sized clusters, with an between-class edge probability of $q$
and a within-class edge probability of $p$.
Although most of the literature on this model has focused on the
case of increasing degrees (ie.\ $pn, qn \to \infty$ as $n \to \infty$),
the sparse case $p, q = O(1/n)$ is interesting both from a mathematical
and an applied point of view.

A striking conjecture of Decelle, Krzkala, Moore and Zdeborov\'a based on deep, non-rigorous ideas from statistical physics
gave a precise prediction for the algorithmic threshold of clustering in the sparse planted partition model.
In particular, if $p = a/n$ and $q = b/n$, then Decelle et al.\ conjectured
that it is possible to cluster in a way correlated with the true partition if
$(a - b)^2 > 2(a + b)$, and impossible if
$(a - b)^2 < 2(a + b)$.
By comparison, the best-known rigorous result is that of Coja-Oghlan, who showed
that clustering is possible if $(a - b)^2 > C (a + b)$ for some sufficiently large $C$.

We prove half of their prediction, showing that
it is indeed impossible to cluster if $(a - b)^2 < 2(a + b)$.
Furthermore we show that it is impossible even to estimate the model
parameters from the graph when $(a - b)^2 < 2(a + b)$;
on the other hand, we provide a simple and efficient algorithm for
estimating $a$ and $b$ when $(a - b)^2 > 2(a + b)$.
Following Decelle et al, our work establishes a rigorous connection
between the clustering problem, spin-glass models on the
Bethe lattice and the so called reconstruction problem. This connection
points to fascinating applications and open problems.
\end{abstract}

\section{Introduction}

\subsection{The planted partition problem}

The clustering problem in its general form is, given a
(possibly weighted) graph,
to divide its vertices into several strongly connected classes
with relatively weak cross-class connections.
This problem is fundamental in modern statistics,
machine learning and data mining, but its applications
range from population genetics~\cite{PriSteDon:00}, where it is used
to find genetically similar sub-populations, to image processing~\cite{ShiMalik:00,SonHlaBoy:98},
where it can be used to segment images or to group similar images,
to the study of social networks~\cite{NewWatStr:02}, where it is used to find
strongly connected groups of like-minded people.

The algorithms used for clustering are nearly as
diverse as their applications.  On one side are the hierarchical
clustering algorithms~\cite{Johnson:67} which build a hierarchy of larger
and larger communities, by either recursive aggregation
or division. On the other hand model-based statistical methods,
including the celebrated EM algorithm~\cite{DemLaiRub:77}, are used to fit
cluster-exhibiting statistical models to the data.
A third group of methods work by optimizing some sort of cost
function, for example by finding a minimum cut~\cite{HartuvShamir:00,ShiMalik:00} or by maximizing
the Girvan-Newman modularity~\cite{BC:09,NewmanGirvan:04}.

Despite the variety of available clustering algorithms,
the theory of clustering contains some fascinating
and fundamental algorithmic challenges. For example,
the ``min-bisection'' problem -- which asks for the
smallest graph cut dividing a graph into two
equal-sized pieces -- is well-known to be NP-hard~\cite{GJS:76}.
Going back to the 1980s, there has been much study of
the average-case complexity of the min-bisection problem.
For instance, the min-bisection problem is much easier
if the minimum bisection is substantially smaller than
most other bisections. This has led to interest in
random graph models for which a typical sample has
exactly one good minimum bisection.
Perhaps the simplest such model is the ``planted bisection''
model, which is similar to the Erd\"os-Renyi model.
\begin{definition}[The planted bisection model]
  For $n \in \N$ and $p, q \in (0, 1)$, let $\curlyG(n, p, q)$
  denote the model of random, $\pm$-labelled graphs in which
  each vertex $u$ is assigned (independently and uniformly at random)
  a label $\sigma_u \in \{\pm\}$, and then each possible edge
  $(u, v)$ is included with probability $p$ if $\sigma_u = \sigma_v$
  and with probability $q$ if $\sigma_u \ne \sigma_v$.
\end{definition}
If $p = q$, the planted partition model is just an
Erd\"os-Renyi model, but if $p \gg q$ then a typical
graph will have two well-defined clusters. Actually,
the literature on the min-bisection problem usually
assumes that the two classes have exactly the same size
(instead of a random size), but
this modification makes almost no difference in the context of
this work.

The planted bisection model was not the earliest model to
be studied in the context of min-bisection -- Bui et
al.~\cite{BCLS:87} and Boppana~\cite{B:87} considered
graphs chosen
uniformly at random from all graphs with a given number
of edges and a small minimum
bisection.
Dyer and Frieze~\cite{DF:89} were the first to study
the min-bisection problem on the planted bisection model;
they showed that if $p > q$ are fixed as $n \to \infty$
then the minimum bisection is the one that separates
the two classes, and it can be found in expected
$O(n^3)$ time.

The result of Dyer and Frieze was improved by Jerrum
and Sorkin~\cite{JS:98}, who reduced the running
time to $O(n^{2 + \epsilon})$ and allowed $p - q$ to
shrink at the rate $n^{-1/6 + \epsilon}$.
More interesting than these improvements, however,
was the fact that Jerrum and Sorkin's analysis
applied to the popular and fast-in-practice
Metropolis algorithm.
Later, Condon and Karp~\cite{CK:01}
gave better theoretical guarantees with
a linear-time algorithm that works for
$p - q \ge \Omega(n^{-1/2 + \epsilon})$.

With the exception of Boppana's work (which was for a
different model), the aforementioned results applied
only
to relatively dense graphs. McSherry~\cite{M:01} showed
that a spectral clustering algorithm works as long
as $p - q \ge \Omega(\sqrt{q (\log n)/n})$.
In particular, his result is meaningful for graphs
whose average degree is as low as $O(\log n)$.
These are essentially the sparsest possible
graphs for which the minimum cut will agree
with the planted bisection, but Coja-Oghlan~\cite{CO:10}
managed to obtain a result for even sparser graphs
by studying a relaxed problem. Instead of trying
to recover the minimum bisection, he showed that
a spectral algorithm will find a bisection which
is positively correlated with the planted bisection.
His result applies as long as $p - q \ge \Omega(\sqrt{q/n})$,
and so it is applicable even to graphs with a constant
average degree.

\subsection{Block Models in Statistics} 
The statistical literature on clustering is more closely focused on real-world network data with 
the planted bisection model (or ``stochastic
blockmodel,'' as it is known in the statistics
community) used as an important test-case
for theoretical results.
Its study goes back to Holland et al.~\cite{HLL:83}, who
discussed parameter estimation and gave
a Bayesian method for finding
a good bisection, without theoretical
guarantees. Snijders and Nowicki~\cite{SN:97}
studied several different statistical methods
-- including
maximum likelihood estimation and the EM algorithm --
for the planted bisection model with $p - q = \Omega(1)$.
They then applied those methods to social networks data.
More recently, Bickel and Chen~\cite{BC:09} showed
that maximizing the Girvan-Newman modularity -- a
popular measure of cluster strength -- recovers the
correct bisection, for the same range of
parameters as the result of McSherry.
They also demonstrated that their methods perform
well on social and telephone network data.
Spectral clustering, the method studied by Boppana
and McSherry, has also appeared in the statistics literature:
Rohe et al.~\cite{RCY:11} gave a theoretical analysis of spectral
clustering under the planted bisection model and also
applied the method to data from Facebook.

\subsection{Sparse graphs and insights from statistical physics} 
The case of sparse graphs with constant average degree is well motivated from the perspective of real networks. Indeed, Leskovec et al.~\cite{LLDM:08}
collected and studied a vast collection of large network
datasets, ranging from social networks like LinkedIn
and MSN Messenger, to collaboration networks in movies
and on the arXiv, to biological networks in yeast.
Many of these networks had millions of nodes, but most
had an average degree of no more than 20; for instance,
the LinkedIn network they studied had approximately
seven million nodes, but only 30 million edges.
Similarly, the real-world networks considered by
Strogatz~\cite{Strogatz:01} -- which include coauthorship
networks, power transmission networks and web link networks
-- also had small average degrees.
Thus it is natural to consider the
planted partition model with parameters $p$ and $q$
of order $O(1/n)$.

Although sparse graphs are natural for modelling
many large networks, the planted partition model
seems to be most difficult to analyze in the sparse
setting. Despite the large amount of work studying
this model, the only results we know of that
apply in the sparse case $p, q = O(\frac{1}{n})$
are those of Coja-Oghlan.
Recently, Decelle et al.~\cite{Z:2011} made some fascinating
conjectures for the cluster identification problem
in the sparse planted partition model. In what follows,
we will set $p = a/n$ and $q = b/n$ for some fixed $a > b > 0$.

\begin{conjecture}\label{conj:reconstruction}
  If $(a - b)^2 > 2 (a + b)$ then the clustering problem
  in $\curlyG(n, \frac an, \frac bn)$
  is solvable as $n \to \infty$, in the sense that one can
  a.a.s.\ find a bisection
  which is positively correlated with the planted bisection.
\end{conjecture}

To put Coja-Oghlan's work into the context of this conjecture,
he showed that if $(a - b)^2 > C (a + b)$ for a large enough
constant $C$, then the spectral method solves the clustering problem.
Decelle et al.'s work is based on deep but non-rigorous ideas
from statistical physics. In order to identify the best bisection,
they use the sum-product algorithm (also known as belief propagation).
Using the cavity method, they argue that the algorithm should work,
a claim that is bolstered by compelling simulation results.

What makes Conjecture~\ref{conj:reconstruction} even more interesting
is the fact that it might represent a threshold for the solvability
of the clustering problem.
\begin{conjecture}\label{conj:non-reconstruction}
  If $(a - b)^2 < 2(a+b)$ then the clustering
  in $\curlyG(n, \frac an, \frac bn)$
  problem is not solvable as $n \to \infty$.
\end{conjecture}
This second conjecture is based on a connection with the
tree reconstruction problem (see~\cite{M:survey} for a survey).
Consider a multi-type branching process where there are two types of particles named $+$
and $-$.
Each particle gives birth to $\Pois(a)$ (ie.\ a Poisson distribution with mean $a$)
particles of the same type and $\Pois(b)$
particles of the complementary type. In the tree reconstruction problem, the goal is to recover the label of the root of the tree
from the labels of level $r$ where $r \to \infty$.
This problem goes back to Kesten and Stigum~\cite{KS} in the 1960s, who
showed that if $(a-b)^2 > 2(a+b)$ then
it is possible to recover the root value with
non-trivial probability. The converse was not resolved until 2000,
when Evans, Kenyon, Peres and Schulman~\cite{EKPS} proved
that if  $(a-b)^2 \leq 2(a+b)$ then it is impossible to recover the root with
probability bounded above $1/2$ independent of $r$.  This is equivalent to the reconstruction or extremality threshold for the Ising model on a branching process.

At the intuitive level the connection between clustering and tree reconstruction, follows from the fact 
that the neighborhood of a vertex in $\curlyG(n, \frac{a}{n}, \frac{b}{n})$
should look like a random labelled tree with high probability. Moreover, the distribution
of that labelled tree should converge as $n \to \infty$ to the
multi-type branching process defined above. We will make this connection formal later.  

Decelle et al.\ also made a conjecture related to the 
the parameter estimation problem that was previously studied extensively in the statistics literature. 
Here the problem is to identify the parameters 
$a$ and $b$. Again, they provided an algorithm based belief propagation and they used physical ideas to argue that
there is a threshold above which the parameters can be estimated,
and below which they cannot.
\begin{conjecture}\label{conj:estimation}
  If $(a - b)^2 > 2 (a + b)$ then there is a consistent estimator
  for $a$ and $b$ under $\curlyG(n, \frac an, \frac bn)$. Conversely,
  if $(a - b)^2 < 2(a+b)$ then there is no consistent estimator.
\end{conjecture}

\section{Our results}

Our main contribution is to establish
Conjectures~\ref{conj:non-reconstruction}
and~\ref{conj:estimation}.

\begin{theorem}\label{thm:non-reconstruction}
If $a + b > 2$
and $(a - b)^2 \le 2(a+b)$ then, for any fixed vertices $u$ and $v$,
\[\P_n(\sigma_u = + \mid G, \sigma_v = +) \to \frac{1}{2} \text{ a.a.s.}\]
\end{theorem}

\begin{remark}
Theorem~\ref{thm:non-reconstruction} is stronger than
Conjecture~\ref{conj:non-reconstruction} because
it says that an even easier problem cannot be solved:
if we take two random vertices of $G$, Theorem~\ref{thm:non-reconstruction}
says that no algorithm can tell whether or
not they have the same label. This is an easier task than finding
a bisection, because finding a bisection is equivalent
to labeling \emph{all} the vertices; we
only asking whether two of them have the same label or not.
Theorem~\ref{thm:non-reconstruction} is also stronger than the
conjecture because it includes the case $(a - b)^2 = 2(a+b)$,
for which Decelle et al.\ did not conjecture any particular
behavior.
\end{remark}
\begin{remark}
Note that the assumption $a + b > 2$ is there to ensure that
$G$ has a giant component, without which the clustering problem
is clearly not solvable.
\end{remark}

To prove Conjecture~\ref{conj:estimation}, we compare
the planted partition model to an appropriate Erd\"os-Renyi
model:
let $\P_n = \curlyG(n, \frac an, \frac bn)$
and take $\P'_n = \curlyG(n, \frac{a+b}{2n})$ to be the Erd\"os-Renyi
model that has the same average degree as $\P_n$.
\begin{theorem}~\label{thm:distinguish}
If $(a - b)^2 < 2(a + b)$ then $\P_n$ and $\P'_n$ are mutually
contiguous i.e., for a sequence of events $A_n$,
$\P_n(A_n) \to 0$ if, and only if, $\P'_n(A_n) \to 0$.

Moreover, if $(a - b)^2 < 2(a + b)$
then there is no consistent estimator for $a$ and $b$. 
\end{theorem}

Note that the second part of the Theorem~\ref{thm:distinguish} follows
from the first part,
since it implies that
$\curlyG(n, \frac an, \frac bn)$ and $\curlyG(n, \frac{\alpha}{n},
\frac{\beta}{n})$ are contiguous as long as
$a + b = \alpha + \beta > \frac{1}{2}\max\{(a-b)^2, (\alpha-\beta)^2\}$.
Indeed one cannot even consistently distinguish the planted partition model
from the corresponding Erd\"os-Renyi model!

The other half of Conjecture~\ref{conj:estimation} follows from
a converse to Theorem~\ref{thm:distinguish}:

\begin{theorem} \label{thm:cycles} 
  If $(a - b)^2 > 2(a + b)$, then $\P_n$ and $\P'_n$ are asymptotically
  orthogonal. 
  Moreover, a consistent estimator for $a,b$ can
  be obtained as follows: let $X_k$ be the number of
  cycles of length $k$, and define
  \begin{align*}
    \hat d_n &= \frac{2|E|}{n} \\
    \hat f_n &= (2k_n X_{k_n} - \hat d_n^{k_n})^{1/{k_n}}
  \end{align*}
  where $k_n = \lfloor \log^{1/4} n \rfloor$. Then
  $\hat d_n + \hat f_n$ is a consistent estimator for $a$ and
  $\hat d_n - \hat f_n$ is a consistent estimator for $b$.
 
 Finally, there is an efficient algorithm whose running time is polynomial in $n$ to calculate $\hat d_n$ and $\hat f_n$.  
\end{theorem}

\subsection{Proof Techniques} 

\subsubsection{Short Cycles} 
To establish Theorem~\ref{thm:cycles}
we count the number of short cycles in $G \sim \P_n$.
It is well-known that the number of
$k$-cycles in a graph drawn from $\P'_n$
is approximately Poisson-distributed with mean
$\frac 1k (\frac{a+b}{2})^k$. Modifying the proof of this result, we will show that 
we will show that the number of
$k$-cycles in $\P_n$ is approximately Poisson-distributed with mean
$\frac 1k \big((\frac{a+b}{2})^k + (\frac{a-b}{2})^k\big)$.

By comparing the first and second moments of Poisson
random variables and taking $k$ to increase slowly with $n$, one can
distinguish between the cycle counts of $G \sim \P_n$ and
$G \sim \P'_n$ as long
as $(a-b)^2 > 2(a+b)$.

The first half of Conjecture~\ref{conj:estimation}
follows because the same comparison of first and second moments
implies that counting cycles gives a consistent estimator for
$a + b$ and $a -b $ (and hence also for $a$ and $b$).

While there is in general no efficient algorithm for counting cycles in graphs, we show that with high probability the number of short cycles coincides with the number of non-backtracking walks of the same length which can be computed efficiently using matrix multiplication. 

The proof of Theorem~\ref{thm:cycles} is carried out in Section~\ref{sec:cycles}. 

\subsubsection{Non-Reconstruction}
As mentioned earlier, Theorem~\ref{thm:non-reconstruction} intuitively follows from the fact 
that the neighborhood of a vertex in $\curlyG(n, \frac{a}{n}, \frac{b}{n})$
should look like a random labelled tree with high probability and the distribution
of that labelled tree should converge as $n \to \infty$ to the
multi-type branching process defined above.  While this intuition is not too hard to justify for small neighborhoods (by proving there are no short cycles etc.) the global ramifications are more challenging to establish. This is
because, that conditioned on the graph structure, the model is neither an Ising model, nor a Markov random field! This is due to two effects: 
\begin{itemize}
\item The fact that the two clusters are of the same (approximate) size. This amounts to a global conditioning on the number of $+/-$'s.
\item The model is not even a Markov random field conditioned on the number of $+$ and $-$ vertices.
  This follows from the fact that for every two vertices $u,v$ that do not form an edge,
  there is a different weight for $\sigma_u = \sigma_v$ and $\sigma_u \neq \sigma_v$.
  In other words, if $a>b$, then there is a slight repulsion (anti-ferromagnetic interaction)
  between vertices not joined by an edge. 
\end{itemize}
In Section~\ref{sec:non_reconstruction}, we prove Theorem~\ref{thm:non-reconstruction} by showing how to overcome the challenges above.

\subsubsection{The Second Moment} 
A major effort is devoted to the proof Theorem~\ref{thm:distinguish}. 
In the proof we show that the random variables $\frac{\P_n(G)}{\P'_n(G)}$ don't have much
mass near 0 or $\infty$. Since the margin of $\P_n$ is somewhat
complicated to work with, the first step is to enrich the distribution
$\P'_n$ by adding random labels. Then we show that the random
variables $Y_n := \frac{\P_n(G, \sigma)}{\P'_n(G, \sigma)}$ don't have mass
near 0 or $\infty$.  Our proof is one of the most elegant applications of the second moment method in the context 
of statistical physics model. We derive an extremely explicit formula for the second moment of $Y_n$ in 
Lemma~\ref{lem:second-mom}. In particular we show that 
\[
\E Y_n^2
= (1 + o(1)) \frac{e^{-t/2 - t^2/4} }{\sqrt{1-t}}, \quad t = \frac{(a-b)^2}{2(a+b)}
\]
This already show that the second moment is bounded off $(a-b)^2 < 2(a+b)$. However, in order to establish the existence of a density, we also need to show that $Y_n$ is bounded away from zero asymptotically. 
In order to establish this, we utilize the small graph conditioning method by calculating joint moments of the number of cycles and $Y_n$. 
It is quite surprising that this calculation can be carried out in rather elegant manner. 

\section{Counting cycles}\label{sec:cycles}

The main result of this section is that the number of $k$-cycles of $G
\sim \P_n$ is approximately Poisson-distributed. We will then use this
fact to show the first part of Theorem~\ref{thm:distinguish}.
Actually, Theorem~\ref{thm:distinguish} only requires us to calculate
the first two moments of the number of $k$-cycles, but the rest of the
moments require essentially no extra work, so we include them for completeness.

\begin{theorem}\label{thm:cycle-poisson}
  Let $X_{k,n}$ be the number of $k$-cycles of $G$, where $G \sim \P_n$.  If
$k = O(\log^{1/4}(n))$ then
\[
  X_{k,n} \toD \Pois\left(\frac{1}{k 2^{k+1}}
  \big((a + b)^k + (a-b)^k\big)\right).
\]
\end{theorem}

Before we prove this, let us explain how it implies 
Theorem~\ref{thm:cycles}. From now on, we will write
$X_k$ instead of $X_{k,n}$.

\begin{proof}[Proof of Theorem~\ref{thm:cycles}]
We start by proving the first statement of the theorem. 
Let's recall the standard fact (which we have mentioned before) that
under $\P'_n$, $X_k \toD \Pois\Big(\frac{(a+b)^k}{k 2^{k+1}}\Big)$. With
this and Theorem~\ref{thm:cycle-poisson} in mind,
\begin{align*}
\E_{\P} X_k, \Var_\P X_k &\to \frac{(a+b)^k + (a - b)^k}{k 2^{k+1}} \\
\E_{\P'} X_k, \Var_{\P'} X_k &\to \frac{(a+b)^k}{k 2^{k+1}}.
\end{align*}

Set $k = k(n) = \log^{1/4} n$ (although any sufficiently slowly
increasing function of $n$ would do).  Choose $\rho$ such that
$\frac{a-b}{2} > \rho > \sqrt{\frac{a+b}{2}}$.  Then $\Var_\P X_k$ and
$\Var_{\P'} X_k$ are both $o(\rho^{2k})$ as $k \to \infty$.
By Chebyshev's inequality,
$X_k \le \E_{\P'} X_k + \rho^k$ $\P'$-a.a.s.\ and $X_k \ge \E_\P X_k -
\rho^k$ $\P$-a.a.s.  Since $\E_\P X_k - \E_{\P'} X_k = \frac{1}{2k}
(\frac{a-b}{2})^k = \omega(\rho^k)$, it follows that $\E_\P X_k -
\rho^k \ge \E_{\P'} X_k + \rho^k$ for large enough $k$.  And so, if we
set $A_n = \{X_{k(n)} \le \E_{\P'} X_{k(n)} + \rho^k\}$ then $\P'(A_n)
\to 1$ and $\P(A_n) \to 0$.

We next show that Theorem~\ref{thm:cycle-poisson}
gives us an estimator for $a$ and $b$ that is consistent
when $(a - b)^2 > 2(a+b)$. First of all, we have a consistent estimator $\hat d$ for 
$d := (a+b)/2$ by simply counting the number of edges.
Thus, if we can estimate $f := (a-b)/2$ consistently then
we can do the same for $a$ and $b$.
Our estimator for $f$ is
\[
  \hat f = (2k X_k - \hat d^k)^{1/k},
\]
where $\hat d$ is some estimator with $\hat d \to d$ $\P$-a.a.s.\ and
$k = k(n)$ increases to infinity slowly enough so that
$k(n) = o(\log^{1/4} n)$ and
$\hat d^k - d^k \to 0$ $\P$-a.a.s.
Take $\sqrt{\frac{a+b}{2}} < \rho < \frac{a-b}{2} = f$;
by Chebyshev's inequality,
$2k X_k - d^k \in [f^k - \rho^k, f^k + \rho^k]$
$\P$-a.a.s.
Since $k = k(n) \to \infty$, $\rho^k = o(f^k)$. Thus,
$2k X_k - d^k = (1 + o(1)) f^k$ $\P$-a.a.s.
Since $\hat d^k - d^k \to 0$
and $f > 1$,
$2k X_k - \hat d^k = f^k + o(1) = (1 + o(1)) f^k$ $\P$-a.a.s.\ and so
$\hat f$ is a consistent estimator for $f$. Finally we take $\hat a = \hat d + \hat f$ and $\hat b = \hat d - \hat f$. 
\end{proof}

\begin{proposition} 
Let $k(n) = o(\log^{1/4} n)$. There is an algorithm whose running time is $O(n \log^{O(1)} n)$
for calculating $\hat a$ and $\hat b$.
\end{proposition}

\begin{proof} 
  Recal $\hat f$ and $\hat d$ from the proof of Theorem~\ref{thm:cycles}.
Clearly, we can compute $\hat d$ in time which is linear in the number of edges.
Thus, we need to show how to find 
$X_k$ in time $O(n \log^{O(1)} n)$. It is easy to see that with high probability,
each neighborhood of radius $2 k(n)$ contains at most one cycle.
Thus, the number of cycles of length $k$ is the same as $\sum_v C_v$, where
$C_v$ is the number of non backtracking walks of length $k$ that start and end at $v$. 

To calculate $C_v$, let $B(v,k)$ be the radius $k$ ball around $v$ in $G$. Let $D_v$ be a diagonal matrix such that for each vertex 
$w \in  B(v,k)$, the diagonal entry corresponding to $w$ is the degree of $w$ in $B(v,r)$. 
Let $A_v$ be the adjacency matrix of $B(v,r)$. It is easy to see that w.h.p. for each $v$, 
$A_v, D_v$ can be generated in $\log^{O(1)} n$ time. Now define $A_{v,0} = I,
A_{v,1} = A_v$ and $A_{v,j} = A_{v,j-1} A_v - D_v A_{v,j-2}$.
Then it is easy to see that the $(v,v)$ entry of $A_{v,k}$ is the number of non-backtracking walks from $v$ to $v$ of length $k$. 
The proof follows. 
\end{proof}

Now we will prove Theorem~\ref{thm:cycle-poisson} using the method of
moments. Recall, therefore, that if $Y \sim \Pois(\lambda)$
then $\E Y_{[m]} = \lambda^m$, where $Y_{[m]}$ denotes the falling factorial
$Y (Y-1) \cdots (Y - m + 1)$. It will therefore be our goal
to show that $\E (X_k)_{[m]} \to \Big(\frac{(a+b)^k + (a - b)^k}{k 2^{k+1}}\Big)^m$.
It turns out that this follows almost entirely from the corresponding
proof for the Erd\"os-Renyi model. The only additional work we
need to do is in the case $m=1$.

\begin{lemma}\label{lem:mean-cycles}
If $k = o(\sqrt n)$ then
\[
\E_{\P} X_k = \binom{n}{k} \frac{(k-1)!}{2}
(2n)^{-k} \Big((a + b)^k + (a - b)^k\Big)
\sim \frac{1}{k 2^{k+1}}\left((a + b)^k + (a-b)^k\right).
\]
\end{lemma}

\begin{proof}
Let $v_0, \dots, v_{k-1}$ be distinct vertices.  Let $Y$ be the
indicator that $v_0 \dots v_{k-1}$ is a cycle in $G$. Then $\E_\P X_k
= \binom{n}{k} \frac{(k-1)!}{2} \E_\P Y$, so let us compute $\E_\P Y$.
Define $N$ to be the number of times in the cycle $v_1 \dots v_k$ that
$\sigma_{v_i} \ne \sigma_{v_{i+1}}$ (with addition taken modulo $k$).
Then
\[
\E_\P Y = \sum_{m=0}^k \P(N = m) \P((v_1 \cdots v_k) \in G | N = m)
= n^{-k} \sum_{m=0}^k \P(N = m) a^{k-m} b^{m}.
\]
On the other hand, we can easily compute $P(N = m)$: for each $i = 0,
\dots, k-2$, there is probability $\frac{1}{2}$ to have $\sigma_{v_i}
= \sigma_{v_{i+1}}$, and these events are mutually indepedent. But
whether $\sigma_{v_{k-1}} = \sigma_{v_0}$ is completely
determined by the other events since there must be an even number of
$i \in \{0, \dots, k-1\}$ such that $\sigma_{v_i} \ne
\sigma_{v_{i+1}}$.  Thus,
\begin{multline*}
\P(N = m) = \Pr\Big(\Binom\Big(k-1, \frac{1}{2}\Big) \in
\{m-1, m\}\Big) \\
= 2^{-k+1} \bigg(\binom{k-1}{m-1} + \binom{k-1}{m}\bigg)
= 2^{-k+1} \binom{k}{m}
\end{multline*}
 for even
$m$, and zero for odd $m$. Hence,
\begin{align*}
\E_\P Y &= n^{-k} 2^{-k+1} \sum_{m \text{ even}} a^{k-m} b^m
\binom{k}{m} \\
&= n^{-k} 2^{-k} \Big((a + b)^k + (a - b)^k\Big).
\end{align*}

The second part of the claim amounts to saying that $n_{[k]} \sim
n^k$, which is trivial when $k = o(\sqrt n)$.
\end{proof}

\begin{proof}[Proof of Theorem~\ref{thm:cycle-poisson}]
Let $\mu = \frac{1}{k2^k} \big((a+b)^k + (a-b)^k\big)$; our
goal, as discussed before Lemma~\ref{lem:mean-cycles},
is to show that $\E(X_k)_{[m]} \to \mu^m$.  Note that
$(X_k)_{[m]}$ is the number of ordered $m$-tuples of $k$-cycles in
$G$. We will divide these $m$-tuples into two sets: $A$ is the set of
$m$-tuples for which all of the $k$-cycles are disjoint, while $B$ is
the set of $m$-tuples in which at least one pair of cycles is not
disjoint.

Now, take $(C_1, \dots, C_m) \in A$. Since the $C_i$ are disjoint,
they appear independently in $G$. By the proof of
Lemma~\ref{lem:mean-cycles}, the probability that cycles $C_1, \dots,
C_m$ are all present is
\[
n^{-km} 2^{-km} \left((a + b)^{k} + (a-b)^{k}\right)^m.
\]
Since there are $\binom{n}{km} \frac{(km)!}{k^m}$ elements of $A$, it
follows that the expected number of vertex-disjoint $m$-tuples of
$k$-cycles is
\[
\binom{n}{km} \frac{(km)!}{k^m}
n^{-km} 2^{-km} \left((a + b)^{k} + (a-b)^{k}\right)^m \sim \mu^m.
\]

It remains to show, therefore, that the expected number of
non-vertex-disjoint $m$-tuples converges to zero. Let $Y$ be the
number of non-vertex-disjoint $m$-tuples,
\[
Y = \sum_{(C_1, \dots, C_m) \in B} \prod_{i=1}^m 1_{\{C_i \subset G\}}.
\]
Then the distribution of $Y$ under $\P$ is stochastically dominated by
the distribution of $Y$ under the Erd\"os-Renyi model
$\curlyG(n, \frac{\max\{a, b\}}{n})$.
It's well-known (see, eg.~\cite{B:random-graphs}, Chapter 4) that
as long as $k = O(\log^{1/4} n)$, $\E
Y \to 0$ under $\curlyG(n, \frac{c}{n})$ for any $c$; hence $\E Y \to
0$ under $\P$ also.
\end{proof}

\section{Non-reconstruction} \label{sec:non_reconstruction}

The goal of this section is to prove
Theorem~\ref{thm:non-reconstruction}. As we said in the introduction,
the proof of Theorem~\ref{thm:non-reconstruction} uses a connection
between $\curlyG(n, \frac an, \frac bn)$ and Markov processes
on trees. Before we go any further, therefore, we should define
a Markov process on a tree and state the result that we will use.

Let $T$ be an infinite rooted tree with root $\rho$. Given
a number
$0 \le \epsilon < 1$,
we will define a random
labelling $\tau \in \{\pm\}^T$. First, we draw $\tau_\rho$
uniformly in $\{\pm\}$. Then, conditionally independently given
$\tau_\rho$, we take every child $u$ of $\rho$ and
set $\tau_u = \tau_\rho$ with probability $1 - \epsilon$ and
$\tau_u = -\tau_\rho$ otherwise. We can continue this construction
recursively to obtain a labelling $\tau$ for which every
vertex, independently, has probability $1-\epsilon$ of having
the same label as its parent.

Back in 1966, Kesten and Stigum~\cite{KS} asked
(although they used somewhat different terminology) whether the label
of $\rho$ could be deduced from the labels of vertices at level
$R$ of the tree (where $R$ is very large). There are many
equivalent ways of stating the question. The interested reader
should see the survey~\cite{M:survey}, because we will only
mention two of them.

Let $T_R = \{u \in T: d(u, \rho) \le R\}$ and define
$\partial T_R = \{u \in T: d(u, \rho) = R\}$. We will write
$\tau_{T_R}$ for the configuration $\tau$ restricted to $T_R$.

\begin{theorem}\label{thm:markov}
Suppose $T$ is a Galton-Watson tree where the offspring distribution
has mean $d > 1$. Then
\[
\lim_{R \to \infty} \Pr(\tau_\rho = + | \tau_{\partial T_R}) = \frac{1}{2}
\text{ a.s.}
\]
if, and only if $d (1-2\epsilon)^2 \le 1$.
\end{theorem}

In particular, if $d(1-2\epsilon)^2 \le 1$ then $\tau_{\partial T_R}$
contains no information about $\tau_\rho$.  Theorem~\ref{thm:markov}
was established by several authors over the course of more than 30
years. The non-reconstruction regime (ie.\ the case
$d (1 - 2\epsilon)^2 \le 1$) is the harder one, and that part
of Theorem~\ref{thm:markov} was first proved for $d$-ary trees in~\cite{BRZ},
and for Galton-Watson trees in~\cite{EKPS}. This latter work
actually proves the result for more general trees in terms of their branching number.

We will be interested in trees $T$ whose offspring distribution is
$\Pois(\frac{a+b}{2})$
and we will take $1-\epsilon = \frac{a}{a+b}$.
Some simple arithmetic applied to Theorem~\ref{thm:markov} then shows
that reconstruction of the root's label
is impossible whenever $(a-b)^2 \le 2(a+b)$.
Not coincidentally, this is the same threshold that
appears in Theorem~\ref{thm:non-reconstruction}.

\subsection{Coupling of balls in $G$ to the broadcast process on trees}

The first step in applying Theorem~\ref{thm:markov} to our
problem is to observe that a neighborhood of $(G, \sigma)
\sim \curlyG(n, \frac an, \frac bn)$ looks
like $(T, \tau)$. Indeed, fix $\rho \in G$
and let $G_R$ be the induced subgraph on $\{u \in G: d(u, \rho)\le R\}$.

\begin{proposition}\label{prop:coupling}
Let $R = R(n) = \lfloor \frac{1}{10\log(2(a + b))} \log n \rfloor$.
There exists a coupling between $(G, \sigma)$
and $(T, \tau)$ such that
$(G_R, \sigma_{G_R}) = (T_R, \tau_{T_R})$ a.a.s.
\end{proposition}

For the rest of this section, we will
take $R = \lfloor \frac{1}{10 \log(2(a + b))} \log n \rfloor$.

The proof of this lemma essentially follows from the fact
that $(T, \tau)$ can be constructed from a sequence of independent
Poisson variables, while $(G_R, \sigma_{G_R})$ can be constructed
from a sequence of binomial variables, with approximately the same
means.

For a vertex $v \in T$, let $Y_v$ be the number of children of $v$;
let $Y_v^=$ be the number of children whose label is $\tau_v$ and let
$Y_v^\ne = Y_v - Y_v^=$.  By Poisson thinning, $Y_v^= \sim
\Pois(a/2)$, $Y_v^\ne \sim \Pois(b/2)$ and they are independent. Note
that $(T, \tau)$ can be entirely reconstructed from the label of the
root and the two sequences $(Y_i^=)$,
$(Y_i^\ne)$.

We can almost do the same thing for $G_R$, but it is a little more
complicated.  We will write $V = V(G)$ and $V_R = V(G) \setminus
V(G_R)$.  For every subset $W \subset V$, denote by $W^+$ and $W^-$
the subsets of $W$ that have the corresponding label. For example,
$V_R^+ = \{v \in V_R: \sigma_v = +\}$. For a vertex $v \in \partial
G_R$, let $X_v$ be the number of neighbors that $v$ has in $V_r$; then
let $X_v^=$ be the number of those neighbors whose label is $\sigma_v$
and set $X_v^\ne = X_v - X_v^=$. Then $X_v^= \sim
\Binom(|V_r^{\sigma_v}|, a)$, $X_v^\ne \sim \Binom(|V_r^{-\sigma_v}|, b)$
and they are independent. Note, however, that they do not contain
enough information to reconstruct $G_R$: it's possible to have $u, v
\in \partial G_r$ which share a child in $V_r$, but this cannot be
determined from $X_u$ and $X_v$. Fortunately, such events are very
rare and so we can exclude them. In fact, this process of carefully
excluding bad events is all that needs to be done to prove
Proposition~\ref{prop:coupling}.

In order that we can exclude their complements, let us give names
to all of our
good events. For any $r$, let $A_r$ be the event that no vertex
in $V_{r-1}$ has more than one neighbor in $G_{r-1}$. Let $B_r$ be
the event that there are no edges within $\partial G_r$.
Clearly, if $A_r$ and $B_r$ hold for all $r = 1, \dots, R$ then
$G_R$ is a tree. In fact, it's easy to see that $A_r$ and $B_r$
are the only events that prevent $\{X_v^=, X_v^\ne\}_{v \in G}$
from determining $(G_R, \sigma_{G_R})$.

\begin{lemma}\label{lem:induction}
If \begin{enumerate}
\item $(T_{r-1}, \tau_{T_{r-1}}) = (G_{r-1}, \sigma_{G_{r-1}})$;
\item $X_u^= = Y_u^=$ and $X_u^\ne = Y_u^\ne$ for every
$u \in \partial G_{r-1}$; and
\item $A_r$ and $B_r$ hold
\end{enumerate}
then $(T_r, \tau_{T_r}) = (G_r, \sigma_{G_r})$.
\end{lemma}

\begin{proof}
The proof is essentially obvious from the construction
of $X_u$ and $Y_u$, but we will be pedantic about it anyway.
The statement $(T_{r-1}, \tau_{T_{r-1}}) = (G_{r-1} \sigma_{G_{r-1}})$
means that there is some graph homomorphism $\phi: G_{r-1} \to T_{r-1}$
such that $\sigma_u = \tau_{\phi(u)}$. If $u \in \partial G_{r-1}$
and $X_u^= = Y_{\phi(u)}^=$ and $X_u^\ne = Y_{\phi(u)}^\ne$
then we can extend $\phi$ to $G_{r-1} \cup \nbr(u)$ while preserving
the fact that $\sigma_v = \tau_{\phi(v)}$ for all $v$. On the
event $A_r$, this extension can be made simultaneously for all
$u \in \partial G_{r-1}$, while the event $B_r$ ensures that
this extension remains a homomorphism.
Thus, we have constructed a label-preserving homomorphism
from $(G_r, \sigma_{G_r})$ to $(T_r, \tau_{T_r})$, which is
the same as saying that these two labelled graphs are equal.

From now on, we will not mention homomorphisms; we will just
identify $u$ with $\phi(u)$.
\end{proof}

In order to complete our coupling, we need to identify one more
kind of good
event.
Let $C_r$ be the event
\[
C_r = \{|\partial G_s| \le 2^s(a + b)^s \log n \text{ for all }
s \le r + 1\}.
\]
The events $C_r$ are useful because they guarantee that
$V_r$ is large enough for the desired binomial-Poisson approximation
to hold. The utility of $C_r$ is demonstrated by the next two lemmas.

\begin{lemma}\label{lem:C_r}
  For all $r \le R$,
\[\P(C_r | C_{r-1}, \sigma) \ge 1 - n^{-\log(4/e)}.\]
Moreover, $|G_r| = O(n^{1/8})$ on $C_{r-1}$.
\end{lemma}

\begin{lemma}\label{lem:A_r-B_r}
For any $r$,
\begin{align*}
\P(A_r | C_{r-1}, \sigma) &\ge 1 - O(n^{-3/4}) \\
\P(B_r | C_{r-1}, \sigma) &\ge 1 - O(n^{-3/4}).
\end{align*}
\end{lemma}

\begin{proof}[Proof of Lemma~\ref{lem:C_r}]
First of all, $X_v$ is stochastically dominated by
$\Binom(n, \frac{a + b}{n})$ for any $v$. On
$C_{r-1}$, $|\partial G_{r}| \le 2^r(a+b)^r \log n$ and so
$|\partial G_{r+1}|$ is stochastically dominated by
\[
Z \sim \Binom\Big(2^r(a + b)^r n\log n, \frac{a+b}{n}\Big).
\]
Thus,
\begin{align*}
  \P(\neg C_r \mid C_{r-1}, \sigma) &=
\P\big(|\partial G_{r+1}| > 2^{r+1} (a+b)^{r+1} \log n \big\mid C_{r-1}, \sigma\big) \\
& \le
\P(Z \ge 2 \E Z) \le \left(\frac e4\right)^{\E Z}
\end{align*}
by a multiplicative version of Chernoff's inequality.
But
\[
\E Z = 2^r (a+b)^{r+1} \log n \ge \log n,
\]
which proves the first part of the lemma.

For the second part, on $C_{r-1}$
\[
|G_r| = \sum_{r=1}^R |\partial G_r|
\le \sum_{r=1}^R 2^r(a+b)^r \log n \le (2(a+b))^{R+1}\log n = O(n^{1/8}).
\qedhere
\]
\end{proof}

\begin{proof}[Proof of Lemma~\ref{lem:A_r-B_r}]
For the first claim, fix $u, v \in \partial G_r$. For any $w \in
V_r$, the probability that $(u, w)$ and $(v, w)$ both appear is
$O(n^{-2})$. Now, $|V_r| \le n$ and Lemma~\ref{lem:C_r} implies
that $|\partial G_r|^2 = O(n^{1/4})$. Hence the result follows
from a union bound over all triples $u, v, w$.

For the second part, the probability of having an edge
between any particular $u, v \in \partial G_r$ is $O(n^{-1})$.
Lemma~\ref{lem:C_r} implies that $|\partial G_r|^2 = O(n^{1/4})$
and so the result follows from a union bound over all pairs
$u, v$.
\end{proof}

The final ingredient we need is a bound on the total
variation distance between binomial and Poisson random variables.

\begin{lemma}\label{lem:binom-pois}
If $m$ and $n$ are positive integers then
\[
\Big\|\Binom\Big(m, \frac cn \Big) - \Pois(c)\Big\|_{TV} =
O\Big(\frac{\max\{1, |m-n|\}}{n}\Big).
\]
\end{lemma}

\begin{proof}
  Assume that $m \le 2n$, or else the result is trivial.
A classical result of Hodges and Le Cam~\cite{HlC} shows that
\[
\Big\|\Binom\Big(m, \frac cn \Big) - \Pois\Big(
\frac{mc}{n}\Big)\Big\|_{TV} \le \frac{c^2 m}{n^2} = O(n^{-1}).
\]
With the triangle inequality in mind, we need only show
that $\Pois(cm/n)$ is close to
$\Pois(c)$. This follows from a direct computation:
if $\lambda < \mu$ then $\big\| \Pois(\lambda)
- \Pois(\mu)\big\|_{TV}$ is just
\[
\sum_{k \ge 0} \frac{|e^{-\mu} \mu^k - e^{-\lambda} \lambda^k|}{k!}
\le
|e^{-\mu} - e^{-\lambda}| \sum_{k \ge 0} \frac{\mu^k}{k!}
+ e^{-\lambda} \sum_{k \ge 0} \frac{|\mu^k - \lambda^k|}{k!}.
\]
Now the first term is $e^{\mu - \lambda} - 1$ and we
can bound $\mu^k - \lambda^k \le k (\mu - \lambda) \mu^{k-1}$
by the mean value theorem. Thus,
\[
\big\| \Pois(\lambda)
- \Pois(\mu)\big\|_{TV}
\le e^{\mu - \lambda} - 1 + e^{\mu - \lambda}(\mu - \lambda)
= O(\mu - \lambda).
\]
The claim follows from setting $\mu = c$ and $\lambda
= \frac{cm}{n}$.
\end{proof}

Finally, we are ready to prove Proposition~\ref{prop:coupling}.

\begin{proof}[Proof of Proposition~\ref{prop:coupling}]
Let $\tilde \Omega$ be the event that $\Big||V^+| - |V^-|\Big| \le
n^{3/4}$.  By Hoeffding's inequality, $\P(\tilde \Omega) \to 1$ exponentially fast.

Fix $r$ and suppose that $C_{r-1}$ and $\tilde \Omega$ hold, and that
$(T_r, \tau_r) = (G_r, \sigma_r)$.
Then for each $u \in \partial G_r$, $X_u^=$
is distributed as $\Binom(|V_r^{\sigma_u}|, a/n)$.
Now,
\[\frac{n}{2} + n^{3/4} \ge |V^{\sigma_{u}}| \ge |V_r^{\sigma_u}|
\ge |V^{\sigma_u}| - |G_{r-1}| \ge \frac{n}{2} - n^{3/4} - O(n^{1/8})\]
and so Lemma~\ref{lem:binom-pois} implies that
we can couple $X_u^=$ with $Y_u^=$ such that $\P(X_u^= \ne Y_u^=) =
O(n^{-1/4})$
(and similarly for $X_u^\ne$ and $Y_u^\ne$).
Since $|\partial G_{r-1}| = O(n^{1/8})$ by Lemma~\ref{lem:C_r},
the union bound implies that
we can find a coupling such that
with probability at least $1 - O(n^{-1/8})$, $X_u^= = Y_u^=$ and
$X_u^\ne = Y_u^\ne$ for every $u \in \partial G_{r-1}$.  Moreover,
Lemmas~\ref{lem:C_r} and~\ref{lem:A_r-B_r} imply
$A_r, B_r$ and $C_r$ hold simultaneously with probability at least $1
- n^{-\log (4/e)} - O(n^{-3/4})$.  Putting these all together, we see that
the hypothesis of Lemma~\ref{lem:induction} holds with probability at
least $1 - O(n^{-1/8})$. Thus,
\[
\P\Big(
(G_{r+1}, \sigma_{r+1}) = (T_{r+1}, \tau_{r+1}), C_r \Big|
(G_r, \sigma_r) = (T_r, \tau_r), C_{r-1}\Big) \ge 1 - O(n^{-1/8}).
\]
But $\P(C_0) = 1$ and we can certainly couple
$(G_1, \sigma_1)$ with $(T_1, \tau_1)$. Therefore,
with a union bound over $r = 1, \dots, R$, we see that
$(G_R, \sigma_R) = (T_R, \tau_R)$
a.a.s.
\end{proof}

\subsection{No long range correlations in $G$}

We have shown that a neighborhood in $G$ looks like a Galton-Watson
tree with a Markov process on it. In this section, we will apply this
fact to prove Theorem~\ref{thm:non-reconstruction}.  In the statement
of Theorem~\ref{thm:non-reconstruction}, we claimed that
$\E(\sigma_\rho | G, \sigma_v) \to 0$, but this is clearly equivalent
to $\Var(\sigma_\rho | G, \sigma_v) \to 1$. This latter
statement is the one that we will prove, because
the conditional variance has a nice monotonicity property.

The idea behind the proof of Theorem~\ref{thm:non-reconstruction}
is to condition on the labels of $\partial
G_R$, which can only make reconstruction easier.  Then we can remove
the conditioning on $\sigma_v$, because $\sigma_{\partial G_R}$ gives
much more information anyway.  Since Theorem~\ref{thm:markov}
and Proposition~\ref{prop:coupling}
imply that $\sigma_v$ cannot be reconstructed from
$\sigma_{\partial G_R}$, we conclude that it cannot
be reconstructed from $\sigma_v$ either.

The goal of this section is to prove that once we have conditioned on
$\sigma_{\partial G_R}$, we can remove the conditioning on
$\sigma_v$. If $\sigma | G$ were distributed according to a Markov
random field, this would be trivial because conditioning on
$\sigma_{\partial G_R}$ would turn $\sigma_v$ and $\sigma_\rho$
independent. For our model, unfortunately, there are weak long-range
interactions. However, these interactions are sufficiently weak that
we can get an asymptotic independence result for separated sets as
long as one of them takes up most of the graph.

In what follows, we say that $X = o(a(n))$ a.a.s.\ if
for every $\epsilon > 0$, $\Pr(|X| \ge \epsilon a(n)) \to 0$
as $n \to \infty$,
and we say that $X = O(a(n))$ a.a.s.\ if
\[
  \limsup_{K\to \infty} \limsup_{n \to \infty} \Pr(|X| \ge K a(n))
  = 0.
\]

\begin{lemma}\label{lem:independence}
Let $A = A(G), B = B(G), C = C(G) \subset V$
be a (random) partition of $V$ such
that $B$ separates $A$ and $C$ in $G$. If $|A \cup B| = o(\sqrt n)$
for a.a.e.\ $G$
\[
\P(\sigma_A | \sigma_{B \cup C}, G) = (1 + o(1))\P(\sigma_A | \sigma_B, G)
\]
for a.a.e.\ $G$ and $\sigma$.
\end{lemma}

Note that Lemma~\ref{lem:independence}
is only true for a.a.e.\ $\sigma$. In particular,
the lemma does not hold for $\sigma$ that are very unbalanced
(eg.\ $\sigma = +^V$).

\begin{proof}
As in the analogous proof for a Markov random field, we factorize
$\P(G, \sigma)$ into parts depending on $A$, $B$ and $C$. We then
show that the part which measures the interaction between $A$ and
$C$ is negligible.
The rest of the proof is then quite similar
to the Markov random fields case.

Define
\[
\psi_{uv}(G, \sigma) =
\begin{cases}
\frac{a}{n} & \text{if $(u, v) \in E(G)$ and $\sigma_u = \sigma_v$} \\
\frac{b}{n} & \text{if $(u, v) \in E(G)$ and $\sigma_u \ne \sigma_v$} \\
1 - \frac{a}{n} & \text{if $(u, v) \not \in E(G)$ and $\sigma_u = \sigma_v$} \\
1 - \frac{b}{n} & \text{if $(u, v) \not \in E(G)$ and $\sigma_u \ne \sigma_v$.}
\end{cases}
\]

For arbitrary subsets $U_1, U_2 \subset V$, define
\[
Q_{U_1,U_2} = Q_{U_1,U_2}(G, \sigma) = \prod_{u \in U_1, v \in U_2}
\psi_{uv}(G, \sigma).
\]
(If $U_1$ and $U_2$ overlap, the product ranges over all unordered
pairs $(u, v)$ with $u \ne v$; that is, if $(u, v)$ is in the product
then $(v, u)$ is not.)
Then
\begin{equation}\label{eq:factorization}
2^n \P(G, \sigma) = \P(G | \sigma)
= Q_{A \cup B, A \cup B} Q_{B \cup C, C} Q_{A, C}.
\end{equation}
First, we will show that $Q_{A,C}$ is essentially independent of $\sigma$.
Take a deterministic sequence $\alpha_n$ with $\alpha_n / \sqrt n \to \infty$
but $\alpha_n |A| = o(n)$ a.a.s.
Define $s_A(\sigma) = \sum_{v \in A} \sigma_v$ and $s_C(\sigma) = \sum_{v \in C} \sigma_v$
and let
\begin{gather*}
  \Omega = \{\tau \in \{\pm\}^V: |s_C(\tau)| \le \alpha_n\} \\
  \Omega_U = \Omega_U(\sigma) = \{\tau \in \{\pm\}^V: \tau_U = \sigma_U \text{ and } |s_C(\tau)| \le \alpha_n\}.
\end{gather*}
By the definition of $\alpha_n$, if $\tau \in \Omega$ then
$|s_A(\tau) s_C(\tau)| \le |A| \alpha_n = o(n)$ a.a.s.
Thus, $\tau \in \Omega$ implies
\begin{align}
Q_{A,C}(G, \tau)
&= \prod_{u \in A, v \in C} \psi_{uv}(G, \tau) \notag \\
&= \Big(1 - \frac an\Big)^{(|A| |C| + s_A(\tau) s_C(\tau))/2}
\Big(1 - \frac bn\Big)^{(|A| |C| - s_A(\tau) s_C(\tau))/2} \notag \\
&= (1 + o(1)) \Big(1 - \frac an\Big)^{|A| |C| / 2}
\Big(1 - \frac bn\Big)^{|A| |C| / 2}\ \text{a.a.s.}
  \label{eq:term-vanish}
\end{align}
where we have used the fact that $u \in A$, $v \in C$ implies
that $(u, v) \not \in E(G)$, and thus $\psi_{uv}$ is either
$1 - \frac an$ or $1 - \frac bn$. Moreover,
$1 - \frac an$ appears once for every pair $(u, v) \in A \times C$
where $\tau_u = \tau_v$. The number of such pairs is
$|A_+| |C_+| + |A_-| |C_-|$ where $A_+ = \{u \in A: \tau_u = +\}$
(and similarly for $C_+$, etc.); it's easy to check, then,
that $2(|A_+| |C_+| + |A_-| |C_-|) = |A| |C| + s_A s_C$, which
explains the exponents in~\eqref{eq:term-vanish}.

Note that the right hand side of~\eqref{eq:term-vanish}
depends on $G$ (through $A(G)$ and $C(G)$) but not on $\tau$.
Writing $2^{-n} K(G)$ for the right hand side of~\eqref{eq:term-vanish},~\eqref{eq:factorization}
implies that if $\tau \in \Omega$ then
\begin{equation}\label{eq:prob-Qa}
  \P(G, \tau) = (1 + o(1)) K(G) Q_{A \cup B,A \cup B}(G, \tau)
Q_{B \cup C,C}(G, \tau)
\end{equation}
for a.a.e.\ $G$.
Moreover, $\alpha_n / \sqrt n \to \infty$
implies that $\sigma \in \Omega$ for a.a.e.\ $\sigma$, and so
for any $U = U(G)$, $\P(\sigma_U, G) = (1 + o(1)) \P(\sigma_U, \sigma \in \Omega, G)$
a.a.s; therefore,
\begin{align}
  \P(\sigma_U, G)
  &= (1 + o(1)) \P(\sigma_U, \sigma \in \Omega, G) \notag \\
  &= (1 + o(1)) \sum_{\tau \in \Omega_U(\sigma)} \P(\tau, G) \notag \\
  &= (1 + o(1)) K(G) \sum_{\tau \in \Omega_U(\sigma)} Q_{A \cup B,A \cup B}(G, \tau)
  Q_{B \cup C,C}(G, \tau) \label{eq:prob-Q}
\end{align}
for a.a.e.\ $G$ and $\sigma$.
(Note that the $o(1)$ term in~\eqref{eq:prob-Qa} depends only on $G$, so there is no problem in
pulling it out of the sum.)
Applying~\eqref{eq:prob-Q} twice, with $U = A \cup B$ and $U = B$,
\begin{align}
\P(\sigma_A | \sigma_B, G)
&= \frac{\P(\sigma_{A \cup B}, G)}{\P(\sigma_B, G)} \notag \\
&= (1 + o(1))
\frac{\sum_{\tau \in \Omega_{A \cup B}}
Q_{A \cup B,A \cup B}(G, \tau) Q_{B\cup C,C}(G, \tau)}
{ \sum_{\tau \in \Omega_B}
Q_{A \cup B, A \cup B}(G, \tau)
Q_{B \cup C,C}(G, \tau)}.
\label{eq:factorizing-1}
\end{align}
Note that $Q_{U_1,U_2}(\tau)$ depends on $\tau$ only through
$\tau_{U_1 \cup U_2}$.
In particular, in the numerator of~\eqref{eq:factorizing-1},
$Q_{A \cup B, A \cup B}(G, \tau)$ doesn't depend on $\tau$
since we only sum over $\tau$ with $\tau_{A \cup B} = \sigma_{A \cup B}$.
Hence, the right hand side of~\eqref{eq:factorizing-1}
is just
\begin{equation}\label{eq:factorizing-2}
  (1 + o(1))\frac{Q_{A \cup B,A \cup B}(G, \sigma)
  \sum_{\tau \in \Omega_{A \cup B}} Q_{B \cup C,C}(G, \tau)}
  {\Big(\sum_{\tau \in \Omega_{B \cup C}} Q_{A \cup B,A \cup B}(G, \tau)\Big)
  \Big(\sum_{\tau \in \Omega_{A \cup B}} Q_{B \cup C,C}(G, \tau)\Big)},
\end{equation}
where we could factorize the denominator because with $\tau_B$ fixed,
$Q_{A \cup B, A \cup B}$
depends only on $\tau_A$, while
$Q_{B \cup C, C}$ depends only on
$\tau_C$.
Cancelling the common terms, then multiplying top and bottom
by $Q_{B \cup C,C}(G, \sigma)$, we have
\begin{align*}
  \eqref{eq:factorizing-2}
  &=
  (1 + o(1))\frac{Q_{A \cup B,A \cup B}(G, \sigma)}
  {\sum_{\tau \in \Omega_{B \cup C}} Q_{A \cup B,A \cup B}(G, \tau)} \\
  &=
  (1 + o(1)) \frac{Q_{A \cup B,A \cup B}(G, \sigma) Q_{B \cup C,C}(G, \sigma)}
  {\sum_{\tau \in \Omega_{B \cup C}} Q_{A \cup B,A \cup B}(G, \tau) Q_{B\cup C,C}(G, \tau)} \\
  &= (1 + o(1)) \frac{\P(G, \sigma)}{\P(G, \sigma_{B \cup C})} \\
  &= (1 + o(1)) \P(\sigma_A | \sigma_{B \cup C}, G) \text{ a.a.s.}
\end{align*}
where the penultimate line used~\eqref{eq:prob-Q} for the denominator
and~\eqref{eq:prob-Qa} (plus the fact that $\sigma \in \Omega$ a.a.s.)
for the numerator.
On the other hand, recall from~\eqref{eq:factorizing-1}
that $\eqref{eq:factorizing-2} = (1 + o(1)) \P(\sigma_A | \sigma_B, G)$
a.a.s.

\end{proof}

\begin{proof}[Proof of Theorem~\ref{thm:non-reconstruction}]
By the monotonicity of conditional
variances,
\[\Var(\sigma_\rho | G, \sigma_v, \sigma_{\partial G_R}) \le
\Var(\sigma_\rho | G, \sigma_v).\]

Since $|G_R| = o(\sqrt n)$
a.a.s.\ and $v \not \in G_R$ a.a.s, it follows
from Lemma~\ref{lem:independence} that $\sigma_v$ and $\sigma_\rho$
are a.a.s.\ conditionally independent given $\sigma_{\partial G_R}$ and $G$. Thus,
$\Var(\sigma_\rho | G, \sigma_v, \sigma_{\partial G_R}) \to
\Var(\sigma_\rho | G, \sigma_{\partial G_R})$.
Now Proposition~\ref{prop:coupling}
implies that $|\Var(\sigma_\rho | G, \sigma_{\partial G_R})
- \Var(\tau_\rho | T, \tau_{\partial T_R})| \to 0$,
but Theorem~\ref{thm:markov} says that
$\Var(\tau_\rho | T, \tau_{\partial T_R}) \to 1 \text{ a.a.s.}$
and so $\Var(\sigma_\rho | G, \sigma_{\partial G_R}) \to 1$
a.a.s.\ also.
\end{proof}

\section{The Second Moment Argument}

In this section, we will prove 
Theorem~\ref{thm:distinguish}. The general direction of this proof was
already described in the introduction, but let's begin here with a
slightly more detailed overview. Recall that $\P'_n$ denotes the
Erd\"os-Renyi model $\curlyG(n, \frac{a+b}{2n})$. The first thing we
will do is to extend $\P'_n$ to be a distribution on labelled graphs.
In order to do this, we only need to describe the conditional
distribution of the label given the graph. We will take
\[
\P'_n(\sigma | G) = \frac{\P_n(G | \sigma)}{Z_n(G)},
\]
where $Z_n(G)$ is the normalization constant for which this is a
probability.  Now, our goal is to show that $\frac{\P_n(G,
  \sigma)}{\P'_n(G, \sigma)}$ is well-behaved; with our
definition of $\P'_n(\sigma | G)$, we have
\[
\frac{\P_n(G, \sigma)}{\P'_n(G, \sigma)}
= \frac{\P_n(\sigma) Z_n(G)}{\P'_n(G)} = 2^{-n} \frac{Z_n(G)}{\P'_n(G)}.
\]
Thus, Theorem~\ref{thm:distinguish} reduces
to the study of the partition function $Z_n(G)$. To do this,
we will use
the small
subgraph conditioning method. This method was developed by Robinson
and Wormald~\cite{RW:92,RW:94} in order to prove that most
$d$-regular graphs are Hamiltonian, but it has since been applied
in many different settings (see the survey~\cite{Wormald} for
a more detailed discussion). Essentially, the method is useful
for studying a sequence $Y_n(G_n)$ of random variables which
are not concentrated around their means, but which become concentrated
when we condition on the number of short cycles that $G_n$
has. Fortunately for us, this method has been developed into
an easily applicable tool, the application of which only requires
the calculation of some joint moments. The formulation below
comes from~\cite{Wormald}, Theorem~4.1.

\begin{theorem}\label{thm:subgraph}
Fix two sequences of probability distributions $\P'_n$ and $\P_n$
on a common sequence of discrete measure spaces, and let $Y_n =
\frac{\P_n}{\P'_n}$ be the density of $\P_n$ with respect to $\P_n$.
Let $\lambda_k > 0$ and $\delta_k \ge -1$ be real numbers.
For each $n$, suppose that there are random variables
$X_k = X_k(n) \in \N$ for $k \ge 3$ such that
\begin{enumerate}[(a)]
\item For each fixed $m \ge 1$, $\{X_k(n)\}_{k=3}^m$ converge jointly
under $\P'_n$ to independent Poisson variables with means $\lambda_k$;

\item For every $j_1, \dots, j_m \in \N$,
\[
  \frac{\E_{\P'_n} \big(Y_n [X_3(n)]_{j_1} \cdots [X_m(n)]_{j_m}\big)}{\E_{\P'_n} Y_n}
\to \prod_{k=3}^m (\lambda_k(1+\delta_k))^{j_k};
\]

\item \[\sum_{k \ge 3} \lambda_k \delta_k^2 < \infty; \]
\item \[
    \frac{\E_{\P'_n} Y_n^2}{(\E_{\P'_n} Y_n)^2} \to
\exp\left( \sum_{k \ge 3} \lambda_k \delta_k^2\right).
\]
\end{enumerate}

Then $\P'_n$ and $\P_n$ are contiguous.
\end{theorem}

In our application of Theorem~\ref{thm:subgraph} the discussion
at the beginning of this section implies that
$Y_n = Y_n(G) = 2^{-n} \frac{Z_n(G)}{\P'_n(G)}$.
We will take $X_k(n)$
to be the number of $k$-cycles in $G_n$. Thus, condition (a)
in Theorem~\ref{thm:subgraph} is already well-known, with
$\lambda_k = \frac{1}{2k} \big(\frac{a+b}{2}\big)^k$.
This leaves us with three conditions to check. We will start with (d),
but before we do so, let us fix some notation.

Let $\sigma$ and $\tau$ be two labellings in $\{\pm \}^n$.
We will also omit the subscript $n$ in $\P_n$ and $\P'_n$,
and when we write $\prod_{(u, v)}$, we mean that $u$ and $v$ range
over all unordered pairs of distinct vertices $u, v \in G$.
Let $t$ (for ``threshold'') be defined by
$t = \frac{(a-b)^2}{2(a+b)}$.

For the rest of this section, $G \sim \P'$. Therefore
we will drop the $\P'$ from $\E_{\P'}$ and just write $\E$.

\subsection{The first two moments of $Y_n$}

Since $Y_n = \frac{\P(G, \sigma)}
{\P'(G, \sigma)}$, $\E Y_n = 1$ trivially. Let's do a short computation
to double-check it, though, because it will be useful later. Define
\[
W_{uv} = W_{uv}(G, \sigma) =
\begin{cases}
\frac{2a}{a+b} & \text{if $\sigma_u = \sigma_v, \quad  (u, v) \in E$} \\
\frac{2b}{a+b} & \text{if $\sigma_u \neq \sigma_v,\quad (u, v) \in E$} \\
\frac{n-a}{n - (a+b)/2} & \text{if $\sigma_u = \sigma_v, \quad (u, v) \notin E$} \\
\frac{n-b}{n - (a+b)/2} & \text{if $\sigma_u \neq \sigma_v, \quad (u, v) \notin E$}
\end{cases}
\]
and define $V_{uv}$ by the same formula, but with $\sigma$
replaced by $\tau$. Then
\[
Y_n = 2^{-n} \sum_{\sigma \in \{\pm\}^n} \prod_{(u, v)} W_{uv}
\]
and
\[
Y_n^2 = 2^{-2n} \sum_{\sigma, \tau \in \{\pm\}^n} \prod_{(u, v)}
W_{uv} V_{uv}.
\]
Since $\{W_{uv}\}_{(u, v)}$ are independent given $\sigma$,
it follows that
\begin{equation}\label{eq:first-mom}
\E Y_n = 2^{-n} \sum_{\sigma \in \{\pm\}^n} \prod_{(u, v)}
\E W_{uv}
\end{equation}
and
\begin{equation}\label{eq:second-mom}
\E Y_n^2 = 2^{-2n} \sum_{\sigma, \tau \in \{\pm\}^n} \prod_{(u, v)}
\E W_{uv} V_{uv}.
\end{equation}

Thus, to compute $\E Y_n$, we should compute $\E W_{uv}$,
while computing $\E Y_n^2$ involves computing $\E W_{uv} V_{uv}$.

\begin{lemma}\label{lem:first-mom}
For any fixed $\sigma$,
\[
\E W_{uv}(G, \sigma) = 1.
\]
\end{lemma}

\begin{proof}
Suppose $\sigma_u = \sigma_v$. Then $\P'((u, v) \in E) = \frac{a+b}{2n}$, so
\[
\E W_{uv} = \frac{2a}{a+b} \cdot \frac{a+b}{2n}
+ \frac{n-a}{n-(a+b)/2} \cdot \left(1 - \frac{a+b}{2n}\right)
= \frac{a}{n} + 1 - \frac{a}{n} = 1.
\]
The case for $\sigma_u \neq \sigma_v$ is similar.
\end{proof}

Notwithstanding that computing $\E Y_n$ is trivial
anyway, Lemma~\ref{lem:first-mom} and~\eqref{eq:first-mom}
together imply that $\E Y_n = 1$. Let us now move on to the
second moment.

\begin{lemma}\label{lem:cross-moments}
If $\sigma_u \sigma_v \tau_u \tau_v = \plus$ then
\[
\E W_{uv} V_{uv} = 1 + \frac 1n \cdot \frac{(a-b)^2}{2(a+b)}
+ \frac{(a-b)^2}{4n^2} + O(n^{-3}).
\]
If $\sigma_u \sigma_v \tau_u \tau_v = \minus$ then
\[
\E W_{uv} V_{uv} = 1 - \frac 1n \cdot \frac{(a-b)^2}{2(a+b)} -
\frac{(a-b)^2}{4n^2} + O(n^{-3}).
\]
\end{lemma}

\begin{proof}
Suppose $\sigma_u \sigma_v = \tau_u \tau_v = +1$. Then
\begin{align*}
\E W_{uv} V_{uv}
&= \left(\frac{2a}{a+b}\right)^2 \cdot \frac{a+b}{2n}
+ \left(\frac{n - a}{n - (a+b)/2}\right)^2 \cdot \left(1 - \frac{a+b}{2n}\right) \\
&= \frac{2a^2}{n(a+b)} + \frac{(1 - \frac an )^2}{1 - \frac{a+b}{2n}} \\
&= \frac{2a^2}{n(a+b)}
+ \left(1 - \frac an \right)^2
  \left(1 + \frac{a+b}{2n} + \frac{(a+b)^2}{4n^2} + O(n^{-3})\right) \\
&= 1 + \frac 1n \cdot \frac{(a-b)^2}{2(a+b)}
+ \frac{(a-b)^2}{4n^2} + O(n^{-3}).
\end{align*}
The computation for $\sigma_u \sigma_v = \tau_u \tau_v = -1$ is analogous.

Now assume $\sigma_u \sigma_v = +1$ while $\tau_u \tau_v = -1$. By a very similar computation,
\begin{align*}
\E W_{uv} V_{uv}
&= \frac{4ab}{(a+b)^2} \cdot \frac{a+b}{2n}
+ \frac{(1 - \frac an) (1 - \frac bn)}{(1 - \frac{a+b}{2n})^2}
\left(1 - \frac{a+b}{2n}\right) \\
&= 1 - \frac 1n \cdot \frac{(a-b)^2}{2(a+b)} -
\frac{(a-b)^2}{4n^2} + O(n^{-3}).
\end{align*}
The computation for $\sigma_u \sigma_v = -1, \tau_u \tau_v = +1$ is analogous.
\end{proof}

Given what we said just before Lemma~\ref{lem:first-mom},
we can now compute $\E Y_n^2$ just by looking at the number of $(u,v)$ where 
$\sigma_u \sigma_v \tau_u \tau_v = \pm 1$. 
To make this easier,
we introduce another parameter,
$\rho = \rho(\sigma, \tau) = \frac{1}{n} \sum_i \sigma_i \tau_i$.
Writing $s_{\pm}$ for the number of $\{u,v\}$ with $u \neq v$ for which $\sigma_u \sigma_v \tau_u \tau_v = \pm$ we get: 
\[
\rho^2 = n^{-1} + 2 n^{-2} \sum_{u \neq v} \sigma_u \sigma_v \tau_u \tau_v = n^{-1} + 2 n^{-2} (s_+ - s_-)
\]
Since we also have $2 n^{-2}(s_+ + s_-) = 1-n^{-1}$, we obtain 
\begin{align*}
s_+
= (1+\rho^2) \frac{n^2}{4} - \frac{n}{2}, \quad 
s_-
= (1-\rho^2) \frac{n^2}{4}.
\end{align*}

\begin{lemma}\label{lem:second-mom}
\[
\E Y_n^2
= (1 + o(1)) \frac{e^{-t/2 - t^2/4} }{\sqrt{1-t}}.
\]
\end{lemma}

Before we proceed to the proof, recall (or check, by writing
out the Taylor series of the logarithm) that
\[
\left(1 + \frac{x}{n}\right)^{n^2} = (1 + o(1)) e^{nx - \frac{1}{2} x^2}
\]
as $n \to \infty$.

\begin{proof}
Define $\gamma_n = \frac tn + \frac{(a-b)^2}{4n^2}$;
note that
\begin{align*}
(1 + \gamma_n)^{n^2} &= (1 + o(1)) \exp\left(
\frac{(a-b)^2}{4} + tn - \frac{t^2}{2}\right) \\
(1 - \gamma_n)^{n^2} &= (1 + o(1)) \exp\left(
-\frac{(a-b)^2}{4} - tn - \frac{t^2}{2}\right) \\
(1 + \gamma_n)^n &= (1 + o(1)) \exp(t).
\end{align*}
Then, by Lemma~\ref{lem:cross-moments},
\begin{align*}
2^{2n} \E Y_n^2
&= \sum_{\sigma, \tau} \prod_{(u, v)} \E W_{uv} V_{uv} \\
&= \sum_{\sigma, \tau}
(1 + \gamma_n + O(n^{-3}))^{s_+}
(1 - \gamma_n + O(n^{-3}))^{s_-} \\
&= (1 + o(1)) e^{-t/2}
\sum_{\sigma, \tau} (1 + \gamma_n)^{(1 + \rho^2) n^2/4}
(1 - \gamma_n)^{(1 - \rho^2) n^2/4} \\
&= (1 + o(1)) e^{-t/2 - t^2/4}
\sum_{\sigma, \tau} \exp\left(\frac{\rho^2}{2}\left(\frac{(a-b)^2}{4} + tn\right)
\right).
\end{align*}

Computing the last term would be easy if $\rho \sqrt n$ were
normally distributed. Instead, it is binomially distributed,
which -- unsurprisingly -- is just as good. To show it, though, will
require a slight digression.

\begin{lemma}\label{lem:sub-lemma}
If $\xi_i \in \{\pm\}$ are taken uniformly and independently at random
and $Z_n = \frac{1}{\sqrt n}\sum_{i=1}^n \xi_i$ then
\[
\E \exp(s Z_n^2/2) \to \frac{1}{\sqrt{1 - s}}
\]
whenever $s < 1$.
\end{lemma}

\begin{proof}
Since $z \mapsto \exp(sz^2/2)$ is a continuous function,
the central limit theorem implies that
$\exp(s Z_n^2/2) \toD \exp(s Z^2 / 2)$, where $Z \sim \normal(0, 1)$.
Now, $\E \exp(s Z^2 / 2) = \frac{1}{\sqrt{1-s}}$ and so the proof
is complete if we can show that the sequence
$\exp(s Z_n^2/2)$ is uniformly integrable. But this follows from
Hoeffding's inequality:
\[
\Pr(\exp(s Z_n^2/2) \ge M) = \Pr\left(|Z_n| \ge \sqrt{\frac{2 \log M}{s}}
\right) \le M^{-1/s},
\]
which is integrable near $\infty$
(uniformly in $n$) whenever $s < 1$.
\end{proof}

To finish the proof of Lemma~\ref{lem:second-mom}, take $Z_n$
as in Lemma~\ref{lem:sub-lemma} and note that
\[
2^{-2n} \sum_{\sigma, \tau}
\exp\left(\frac{\rho^2}{2}\left(\frac{(a-b)^2}{4} + tn\right)\right)
= \E \exp\left(\frac{t(1 + o(1))}{2}Z_n^2\right)
\to \frac{1}{\sqrt{1 - t}}.
\qedhere
\]
\end{proof}

\subsection{Dependence on the number of short cycles}

Our next task is to check condition (b) in Theorem~\ref{thm:subgraph}.
Note, therefore, that $[X_3]_{j_3} \cdots [X_m]_{j_m}$ is the number of ways
to have an ordered tuple containing $j_3$
3-cycles of $G$, $j_4$ 4-cycles of $G$, and so on.
Therefore, if we can compute $\E Y_n 1_H$ where $1_H$ indicates
that any particular union of cycles occurs in $G_n$, then
we can compute $\E Y_n [X_3]_{m_3} \cdots [X_m]_{j_m}$. Computing
$\E Y_n 1_H$ is the main task of this section; we will do it in three
steps. First, we will get a general formula for
$\E Y_n 1_H$ in terms of $H$. We will apply this general formula
in the case that $H$ is a single cycle and get a much simpler
formula back. Finally, we will extend this to the case when $H$
is a union of vertex-disjoint cycles. 

As promised, we begin the program with a general formula for $\E 1_H
Y_n$. Let $H$ be a graph on some subset of $[n]$, with $|V(H)| =
m$. With some slight abuse
of notation, We write $1_H$ for the random variable that is $1$ when
$H \subset G$, and $\P'(H)$ for the probability that $H \subset G$.

\begin{lemma}\label{lem:cond-moments}
\[
\E 1_H Y_n =
2^{-m} \P'(H) \sum_{\sigma \in \{\pm 1\}^m}
\prod_{(u, v) \in E(H)} w_{uv}(\sigma),
\]
where
\[
w_{uv}(\sigma) =
\begin{cases}
\frac{2a}{a+b} & \text{if $(u, v) \in S(\sigma)$} \\
\frac{2b}{a+b} & \text{otherwise.}
\end{cases}
\]
\end{lemma}

\begin{proof}
We break up $\sigma \in \{\pm 1\}^n$ into $(\sigma_1, \sigma_2) \in
\{\pm 1\}^{V(H)} \times \{\pm 1\}^{V(G) \setminus V(H)}$ and sum over the two parts
separately.  Note that if $(u,v) \in E(H)$ then $W_{uv}(G, \sigma)$
depends on $\sigma$ only through $\sigma_1$. Let $D(H) = E(G) \setminus E(H)$,
so that $(u, v) \in D(H)$ implies
that $W_{uv}$ and $1_H$ are independent. Then
\begin{align*}
\E 1_H Y_n
&= 2^{-n} \sum_{\sigma_1} \sum_{\sigma_2} \E 1_H
\prod_{(u, v)} W_{uv}(G, \sigma) \\
&= 2^{-n} \sum_{\sigma_1} \bigg( \Big(\E 1_H \prod_{(u, v) \in E(H)} W_{uv}\Big)
\sum_{\sigma_2} \prod_{(u, v) \in D(H)} \E W_{uv} \bigg) \\
&= 2^{-m} \sum_{\sigma_1} \Big(\E 1_H \prod_{(u, v) \in E(H)} W_{uv}\Big),
\end{align*}
because if $(u, v) \in D(H)$ then, for every $\sigma$,
Lemma~\ref{lem:first-mom}
says that $\E W_{uv}(G, \sigma) = 1$.
To complete the proof, note that if $(u, v) \in E(H)$
then for any $\sigma$,
$W_{uv}(G, \sigma) \equiv w_{uv}(\sigma)$ on the event $H \subseteq G$.
\end{proof}

The next step is to compute the right hand
side of Lemma~\ref{lem:cond-moments} in the case
that $H$ is a cycle. This computation is very similar to the one
in Lemma~\ref{lem:mean-cycles}, when we computed the expected number
of $k$-cycles in $\curlyG(n, \frac an, \frac bn)$. Essentially,
we want to compute the expected ``weight'' of a cycle, where
the weight of each edge depends only on whether its endpoints
have the same label or not.

\begin{lemma}\label{lem:one-cycle}
If $H$ is a $k$-cycle then
\[
\sum_{\sigma \in \{\pm 1\}^H} \prod_{(u, v) \in E(H)} w_{uv} (\sigma) =
2^k \left(1 + \left(\frac{a-b}{a+b}\right)^k\right).
\]
\end{lemma}

\begin{proof}
Let $e_1, \dots, e_k$ be the edges of $H$.
Provided that we renormalize,
we can replace the sum over $\sigma$ by an expectation,
where $\sigma$ is taken uniformly in $\{\pm 1\}^H$.
Now, let $N$ be the number of edges of $H$
whose endpoints have different labels. As discussed in the proof
of Lemma~\ref{lem:mean-cycles}, $\Pr(N = j) =
2^{-k+1} \binom{k}{j}$ for even $j$, and zero otherwise.
Then
\begin{align*}
\E_\sigma \prod_{(u,v) \in E(H)} w_{uv}(\sigma)
&= \E_\sigma \left(\frac{2a}{a+b}\right)^{k-N}
\left(\frac{2b}{a+b}\right)^N \\
&= \frac{2}{(a+b)^k} \sum_{j \text{ even}}
\binom{k}{j} a^{k-j} b^j \\
&= 1 + \left(\frac{a-b}{a+b}\right)^k.\qedhere
\end{align*}
\end{proof}

Extending this calculation to vertex-disjoint
unions of cycles is quite easy: suppose
$H$ is the union of cycles $H_i$. Since $w_{uv}(\sigma)$
only depends on $\sigma_u$ and $\sigma_v$, we can just
split up the sum over $\sigma \in \{\pm\}^H$ into a product
of sums, where each sum ranges over $\{\pm\}^{H_i}$. Then
applying Lemma~\ref{lem:one-cycle} to each $H_i$ yields
a formula for $H$.

\begin{lemma}\label{lem:disjoint-cycles}
Define
\[
\delta_k = \left(\frac{a-b}{a+b}\right)^k.
\]
If $H = \bigcup_i H_i$ is a vertex-disjoint union of graphs
and each $H_i$
is a $k_i$-cycle, then
\[
\sum_{\sigma \in \{\pm 1\}^H} \prod_{(u,v) \in E(H)} w_{uv}(H, \sigma)
= 2^{|H|} \prod_i (1 + \delta_{k_i}).
\]
\end{lemma}

We we need one last ingredient, which we hinted at earlier,
before we can show condition (b) of Theorem~\ref{thm:subgraph}.
We only know how to exactly compute $\E Y_n 1_H$ when $H$ is a
disjoint union of cycles. Now, most tuples of cycles are
disjoint, but in order to dismiss the contributions from the
non-disjoint unions, we need some bound
on $\E Y_n 1_H$ that holds for all $H$:

\begin{lemma}\label{lem:horrible-bound}
For any $H$,
\[
\sum_{\sigma \in \{\pm 1\}^H} \prod_{(u, v) \in E(H)} w_{uv}(\sigma)
\le 2^{|H| + |E(H)|}.
\]
\end{lemma}
\begin{proof}
\[
w_{uv}(\sigma) \le \frac{2\max\{a,b\}}{a+b} \le 2
\]
for any $i, j, H$ and $\sigma$.
\end{proof}

Finally, we are ready to put these ingredients
together and prove condition (b) of Theorem~\ref{thm:subgraph}.
For the rest of the section, take
$\delta_k = (\frac{a-b}{a+b})^k$ as it was
in Lemma~\ref{lem:disjoint-cycles}.
Also, recall that $\lambda_k = \frac{1}{2k} \big(\frac{a+b}{2}\big)^k$
is the limit of $\E X_k$ as $n \to \infty$.

\begin{lemma}\label{lem:cond-a}
Let $X_k$ be the number of $k$-cycles in $G$.
For any $j_3, \dots, j_m \in \N$,
\[
\E Y_n \prod_{k=3}^m [X_k]_{j_k}
\to \prod_{k=3}^m (\lambda_k (1 + \delta_k))^{j_k}.
\]
\end{lemma}

\begin{proof}
Set $M = \sum_k k m_k$.
First of all,
\[
[X_k]_{j} = \sum_{H_1, \dots, H_j} \prod_i 1_{H_i}
\]
where the sum ranges over all $j$-tuples of distinct
$k$-cycles, and $1_{H}$ indicates the event that the subgraph $H$
appears in $G$. Thus,
\[
\prod_{k=3}^m [X_k]_{j_k} = \sum_{(H_{ki})} \prod_{k=3}^m \prod_{i=1}^{j_k}
 1_{H_{ki}}
= \sum_{(H_{ki})} 1_{\{\bigcup H_{ki}\}},
\]
where the sum ranges over all $M$-tuples of cycles
$(H_{ki})_{k \le m, i \le j_k}$ for which each $H_{ki}$ is an $k$-cycle,
and every cycle is distinct. Let $\curlyH$ be the set of such tuples;
let $A \subset \curlyH$ be the set of such tuples
for which the cycles are vertex-disjoint, and let $B = \curlyH
\setminus A$.
Thus, if $H = \bigcup H_{ki}$ for $(H_{ki}) \in A$, then
\[
\E Y_n 1_H
= \prod_k(1 + \delta_k)^{j_k} \P'(H)
\]
by Lemmas~\ref{lem:cond-moments} and~\ref{lem:disjoint-cycles}.
Note also that standard counting arguments
(see, for example,~\cite{B:random-graphs}, Chapter 4) imply
that $|A| \P'(H) \to \prod_k \lambda_k^{j_k}$.

On the other hand, if $(H_{ki}) \in B$ then
$H := \bigcup_{ki} H_{ki}$ has at most $M-1$ vertices, $M$ edges,
and its number of edges is strictly larger than its number
of vertices. Thus, $\P'(H) \binom{n}{|H|} \to 0$,
so Lemmas~\ref{lem:cond-moments}
and~\ref{lem:horrible-bound} imply that
\[
\sum_{H' \sim H} \E Y_n 1_H \le \P'(H) |H|! \binom{n}{|H|}
2^{M}
\to 0,
\]
where the sum ranges over all ways to make an isomorphic copy of $H$
on $n$ vertices. Since there are only a bounded number of isomorphism
classes in
\[
\big\{ \bigcup_{ki} H_{ki}: (H_{ki}) \in B\big\},
\]
it follows that $\sum_H \E Y_n 1_H \to 0$, where the sum ranges
over all unions of non-disjoint tuples in $\curlyH$.
 Thus,
\begin{align*}
\E Y_n \prod_{k=3}^m [X_k]_{j_k}
&= \E Y_n \left(
\sum_{(H_{ki}) \in A} 1_{\bigcup H_{ki}}
+ \sum_{(H_{ki}) \not \in B} 1_{\bigcup H_{ki}}
\right) \\
&= |A| \P'(H) \prod_k(1 + \delta_k)^{j_k} + o(1) \\
&\to \prod_k(\lambda_k (1+\delta_k))^{j_k}. \qedhere
\end{align*}
\end{proof}

To complete the proof of Theorem~\ref{thm:distinguish},
note that $\delta_k^2 \lambda_k = \frac{t^k}{2k}$. Thus,
$\sum_{k \ge 3} \delta_k^2 \lambda_k = \frac{1}{2} (\log(1 - t) - t - t^2/2)$.
When $t < 1$, this (with Lemma~\ref{lem:second-mom}) proves conditions
(c) and (d) of Theorem~\ref{thm:subgraph}. Since
condition (a) is classical and condition (b) is given by Lemma~\ref{lem:cond-a},
the conclusion of Theorem~\ref{thm:subgraph}
implies the first statement in 
Theorem~\ref{thm:distinguish}.

We finally apply 
the first half of Theorem~\ref{thm:distinguish}  to show 
that no estimator can be consistent when $(a - b)^2 < 2(a+b)$.
In fact, if $\hat a$ and $\hat b$ are estimators for $a$ and $b$
which converge in probability, then their limit when
$(a - b)^2 < 2(a+b)$ depends only on $a+b$. To see this, let
$\alpha, \beta$ be another choice of parameters with
$(\alpha - \beta)^2 < 2(\alpha + \beta)$ and $\alpha + \beta = a + b$;
Let $\Q_n = \curlyG_n(\alpha, \beta)$; take
$a^*$ to be the in-probability limit of $\hat a$ under $\P_n$
and $\alpha^*$ to be its limit under $\Q_n$.
For an arbitrary $\epsilon > 0$,
let $A_n$ be the event $|\hat a - a^*| > \epsilon$;
thus, $\P_n(A_n) \to 0$.
By the first part of Theorem~\ref{thm:distinguish}, $\P'_n(A_n) \to 0$ also.
Since $\alpha + \beta = a + b$, we can apply the first part of 
Theorem~\ref{thm:distinguish} to $\Q_n$, implying that $\Q_n(A_n) \to 0$
and so $\alpha^* = a^*$.
That is, $\hat a$ converges to the same limit under $\Q_n$ and $\P_n$.

\section{Conjectures Regarding Regular Models} \label{sec:regular} 
We briefly discuss how can one define a regular version of the model and what we expect from the behavior of such a model. 
A regular model should satisfy the following properties:
\begin{itemize}
\item The graph $G$ is a.s.\ a simple $d$-regular graph.
\item For each vertex $u$ among the $d$ neighbors it is connected to, it is connected to $\Binom(d,1-\eps)$ vertices $v$ with 
$\sigma_v = \sigma_u$. 
\item Choices at different vertices are (almost) independent. 
\end{itemize}

As is often the case with random regular graphs, the construction is not completely trivial. 
Here are two possible constructions:
\begin{itemize}
\item

Let $\{X_v: v \in V\}$ be a collection of independent
$\Binom(d,1-\eps)$ variables, conditioned on
\[
  \sum_{v: \sigma_v = +} X_v = \sum_{v: \sigma_v = -} X_v
  \quad \text{is even.}
\]
Now the $(+,+)$ edges are defined by sampling a uniform random graph
on $\{v: \sigma_v = +\}$ with
degree distribution given by $\{X_v: \sigma_v = +\}$,
while the $(-, -)$ edges are defined
by sampling a uniform random graph on $\{v: \sigma_v = -\}$ with
degree distribution given by $\{X_v: \sigma_v = -\}$.
To construct the $(+,-)$ edges we take a uniformly random bipartite graph
with left degrees given by $\{d-X_v: \sigma_v = +\}$
and right degrees given by $\{d-X_v: \sigma_v = -\}$. 

\item
The second construction uses a variant of the configuration model. We generate the graph by generating $d$ independent matchings. 
The probability of each matching is proportional to $(1-\eps)^{n_=} \eps^{n_\neq}$, where $n_=$ is the number of edges $(u,v)$ with $\sigma_u=\sigma_v$ points and $n_{\neq}$ is the number of edges $(u,v)$ with $\sigma_u \neq \sigma_v$. 
\end{itemize}

\subsection{Conjectures}
We conjecture that the results of the paper should extend to the models above where the quantity $(a-b)^2/2(a+b)$ is now replaced by 
$(d-1)\theta^2$, where $\theta = 1-2\eps$. Friedman's proof of Alon's conjecture~\cite{Friedman:10} gives
a very accurate information regarding the spectrum of uniformly random $d$-regular graphs.
We propose the following related conjecture. 

\begin{conjecture} \label{conj:eigen}
Assume $(d-1)\theta^2 > 1$. Then there exist an $\delta > 0$, s.t. with high probability,
the second eigenvalue of the graph generated $\lam_2(G)$ satisfies $\lam_2(G) >  2 \sqrt{d-1} + \delta$. 
Moreover, all other eigenvalues of $G$ are smaller than $2\sqrt{d-1}$, and
the eigenvector associated to $\lam_2(G)$ is correlated with the true partition.
\end{conjecture} 

By comparison, the results of~\cite{Friedman:10} imply that for all $\delta > 0$ with high probability,
if $G$ is a uniformly random $d$-regular graph then
$\lam_2(G) <  2 \sqrt{d-1} + \delta$. Thus the result above provides a simple spectral algorithm to distinguish
between the standard random $d$-regular model and the biased $d$-regular model when
$(d-1)\theta^2 > 1$. Moreover, our conjecture also says that
a spectral algorithm can be used to solve the clustering problem.

Below we sketch a proof for part of Conjecture~\ref{conj:eigen}. Specifically,
we will show that if $(d-1) \theta^2 > 1$ then there is an approximate
eigenvalue-eigenvector pair $(\lambda, f)$ (in the sense that $A f \approx \lambda f$ where
$A$ is the adjacencency matrix of $G$)
where $\lambda > 2\sqrt{d-1} + \delta$ and $f$ is correlated with the true partition.
The more difficult part of the conjecture would be to show that all other eigenvalues
are smaller than $2\sqrt{d-1}$. If this were true, it would imply that
$\lam_2(G) \approx \lam$ and that the eigenvector of $\lam_2(G)$ is close to $f$.

\begin{proof}
  We will assume that $G$ satisfies the following two properties:
\begin{itemize}
\item The process around each vertex looks like the Ising model on a $d$ regular tree.
\item Given two different vertices $u,v$,
  the process in neighborhoods of $u$ and $v$ are asymptotically independent. 
\end{itemize}

Let $r$ be a large constant and let $f(v) = \sum \{ \sigma_w : d(w,v) = r\}$. 
Then $\sum_v f(v) = 0$ and it is therefore orthogonal to the leading eigenvector. 
Let $A$ be the adjacency matrix of the graph. We claim that $\| Af - \lam f \|_2$
is much smaller than $\| f \|_2$, where 
$\lam = \theta^{-1} + (d-1) \theta$.
Note that $\lam > 2 \sqrt{d-1}$
if and only if $|\theta| > (d-1)^{-1/2}$. 

Assuming that the neighborhood of $v$ is a $d$-regular tree,
\[
  (Af)(v) = \sum_{w: d(v, w) = r+1} \sigma_w + (d-1) \sum_{w: d(v,w) = r-1} \sigma_w
\]
and so we can write $(Af)(v) - \lam f(v)$ as
\begin{multline}
  \label{eq:Af-lamf}
  Af(v) - \lam f(v) = \left( \sum_{w: d(v,w) = r+1} \sigma_w - \theta (d-1)\sum_{w: d(v,w)=r} \sigma_w \right) \\
  - \theta^{-1} \left( \sum_{w: d(v,w) = r} \sigma_w - \theta (d-1)\sum_{w: d(v,w)=r-1} \sigma_w \right) 
\end{multline}
We can re-arrange the first sum as
\[
  \sum_{\{w: d(v,w) = r\}} \sum_{\{w' \sim w, d(w', v) = r+1\}} \sigma_w - \theta \sigma_{w'}.
\]
Noting that all the summands are independent given $\{\sigma_w : d(v, w) = r\}$, we see that the above
sum has expectation zero and variance of the order $C (d-1)^r$ for some constant $C$.
Applying a similar decomposition (but at level $r-1$) to the second sum in~\eqref{eq:Af-lamf},
we get
\[
\E[(Af(v)-\lam f(v))^2] \leq C (d-1)^r.
\]
Summing over all $v$, we conclude that
\[
\E[\| Af-\lam f \|_2^2] \leq C n (d-1)^{r} .
\]
On the other hand, from~\cite{HM:98} it follows that for each $v$ individually 
\[
\E[f(v)^2] \geq C' ((d-1) \theta)^{2r},
\]
for some absolute constant $C'$. 
Since the value of $f(v)$ and $f(w)$ for $v \neq w$ are essentially independent,
it follows that with high probability
$\|f\|_2^2 > C' n ((d-1) \theta)^{2r}$. Taking $r$ sufficiently large we see
that $\| Af - \lam f \|_2 \leq \delta(r) \|f\|_2$ with high probability where 
$\delta(r) \to 0$ as $r \to \infty$.
\end{proof}

\section{Open problems} \label{sec:open}

Of the conjectures that we mentioned in the introduction,
Conjecture~\ref{conj:reconstruction} remains open. However,
there are variations and extensions of
Conjectures~\ref{conj:reconstruction}--\ref{conj:estimation}
that may be even more interesting.
For example, we could ask whether Conjecture~\ref{conj:reconstruction}
can be realized by one of several popular and efficient algorithms.

\begin{conjecture}
  \begin{enumerate}
    \item
      If $(a - b)^2 > 2(a + b)$ then the clustering problem
      in $\curlyG(n, \frac an, \frac bn)$ can be solved by
      a spectral algorithm.
    \item
      If $(a - b)^2 > 2(a + b)$ then the clustering problem
      in $\curlyG(n, \frac an, \frac bn)$ can be solved by
     the belief propogation algorithm of~\cite{Z:2011}.
    \item
      If $(a - b)^2 > 2(a + b)$ then the clustering problem
      in $\curlyG(n, \frac an, \frac bn)$ can be solved by
      simulating an Ising model on $G$, conditioned to be almost balanced.
  \end{enumerate}
\end{conjecture}

Of these conjectures, part 2 is closely related to the work
of Coja-Oghlan~\cite{CO:10}, while part 3 would substantially
extend the result of Dyer and Frieze~\cite{DF:89}.

Another way to extend
Conjectures~\ref{conj:reconstruction}--\ref{conj:estimation}
would be to increase the number of clusters from two to $k$.
The model $\curlyG(n, p, q)$ is well-studied for more than
two clusters, in which case it is known as the ``planted
partition'' model.
In fact, many of the results that we cited in the introduction
extend to $k > 2$ also. However, the work of~\cite{Z:2011} suggests
that the case of larger $k$ is rather more delicate than the case
$k = 2$, and that it contains interesting connections to complexity
theory.
The following conjecture comes from their work, and it is
based on a connection to phase transitions in the Potts model
on trees:
\begin{conjecture}\label{conj:k}
  For any $k$, there exists $c(k)$ such that
  if $a > b$ then:
  \begin{enumerate}
    \item if $\frac{(a-b)^2}{a + (k-1) b} < c(k)$ then the
      clustering problem cannot be solved;
    \item if $c(k) < \frac{(a-b)^2}{a + (k-1) b} < k$ then the
      clustering problem is solvable, but not
      in polynomial time;
    \item if $\frac{(a-b)^2}{a + (k-1) b} > k$ then the
      clustering problem can be solved in polynomial time.
  \end{enumerate}
  When $k \le 4$, $c(k) = k$ and so case 2 does not occur.
  When $k \ge 5$, $c(k) < k$.
\end{conjecture}

Part of the difficulty in studying Conjecture~\ref{conj:k}
can be seen from work of the third author~\cite{Sly:11}.
His work contains the best known non-reconstruction results for
the Potts model on trees, but the results for $k > 2$ are less
precise and more difficult to prove than what is known for
$k = 2$.

Decelle et al.\ also state a version of Conjecture~\ref{conj:k} in the case $a < b$. Although
this case is not naturally connected to clustering, it has
close connections to random Boolean satisfiability problems and to spin glasses. In particular,
they conjecture that when $a < b$, case 2 above becomes much larger.

\subsection{Acknowledgments}  A.S. would like to thank Christian Borgs for suggesting the problem and Lenka Zdeborov{\'a} for useful discussions.  Part of this work was done while A.S. was at Microsoft Research, Redmond. The authors would also like to thank
Lenka Zdeborov{\'a} for comments on a draft of this work.
\bibliography{block-model}
\bibliographystyle{plain}
\end{document}